\documentclass[10pt, a4paper, leqno]{article}
\usepackage[intlimits]{amsmath}
\usepackage{amsthm,amssymb,amsfonts,fancyhdr,verbatim,graphicx,bbold,color}

\numberwithin{equation}{section}
\theoremstyle{plain}
\newtheorem{thm}{Theorem}
\newtheorem{lem}[thm]{Lemma}

\theoremstyle{remark}

\newcommand{\Rbb}{{\mathbb{R}}}
\newcommand{\Cbb}{{\mathbb{C}}}
\newcommand{\Nbb}{{\mathbb{N}}}
\newcommand{\spec}{{\mathrm{spec}}}
\newcommand{\ACl}{\mathrm{AC}_{\mathrm{loc}}}
\newcommand{\AC}{\mathrm{AC}}
\newcommand{\Dmax}{\mathcal{D}_\mathrm{max}}
\newcommand{\sch}{\mathcal{C}}
\newcommand{\Lper}{L_\varepsilon}
\newcommand{\DomL}{\mathrm{Dom}(\Lper)}
\newcommand{\spa}{\mathrm{span}}

\newcommand{\be}{\begin{equation}}
\newcommand{\ee}{\end{equation}}
\newcommand{\beu}{\begin{equation*}}
\newcommand{\eeu}{\end{equation*}}
\newcommand{\besu}{\begin{equation*}
\begin{aligned}}
\newcommand{\eesu}{\end{aligned}
\end{equation*}}
\newcommand{\bes}{\begin{equation}
\begin{aligned}}
\newcommand{\ees}{\end{aligned}
\end{equation}}

\newcommand{\Rr}{\mathbb{R}}
\newcommand{\Nn}{\mathbb{N}}
\newcommand{\Cc}{\mathbb{C}}
\newcommand{\loc}{\mathrm{loc}}
\newcommand{\CT}{\mathcal{AT}}
\newcommand{\dom}{\operatorname{Dom}}
\newcommand{\domlper}{\dom(\Lper)}
\renewcommand{\chi}{\mathbb{1}}
\newcommand{\elle}{\ell_\varepsilon}

\newcommand{\Imi}{(-\pi,0)}
\newcommand{\Ipl}{(0,\pi)}
\newcommand{\Ito}{(-\pi,\pi)}

\title{On the stability of a forward-backward heat equation}

\date{31 October 2011}

\author{Lyonell Boulton$^1$ \and Marco Marletta$^2$ \and David Rule$^1$}

\begin{document}
\maketitle

\footnotetext[1]{Department of Mathematics and Maxwell Institute for the Mathematical Sciences, Heriot-Watt University, Edinburgh, EH14 4AS, UK.}

\footnotetext[2]{Cardiff School of Mathematics, Cardiff University,
Senghennydd Road, Cardiff CF24 4AG, UK.}

\begin{abstract}
In this paper we examine spectral properties of a family of periodic 
singular Sturm-Liouville problems which are highly non-self-adjoint but 
have purely real spectrum. The problem originated from the study of the lubrication approximation of a viscous fluid film in the inner surface of a rotating cylinder and has received a substantial amount of attention in recent years. Our main focus will be the determination of Schatten 
class inclusions for the resolvent operator and regularity properties 
of the associated evolution equation. 
\end{abstract}

\tableofcontents

\section{Introduction} \label{section1}
Let $f(x)=\sin(x)$, $c=f'(0)$ and  $0 < \varepsilon < 2/c$. The forward-backward heat equation
\be   \label{fb_eq}
    \left\{
    \begin{aligned} 
    &\partial_t u(t,x) +\elle [u](t,x)=0 \qquad&  x\in\Ito,\, t\in [0, T) \\
    & u(0,x)=g(x) & x\in\Ito \\
    & u(t,-\pi)= u(t,\pi) & t\in[0,T)
    \end{aligned} \right.
\ee
associated to the singular Sturm-Liouville differential operator
\beu
\elle[u](x) = \varepsilon(fu')'(x) +u'(x)   
\eeu
originated in applications from hydrodynamics \cite{2003Benilov}, and it
has recently attracted some interest due to various unusual stability and symmetry properties. Spectral properties of $\elle$  were examined simultaneously in various works by Chugunova, Karabash, Pelinovsky and Pyatkov \cite{2008Chugunovaetal,2009Chugunovaetal}, and Davies and Weir \cite{2007Davies2,2009Daviesetal,2009Weir,2009Weir2,2009Weir3}.  Remarkably it was noted that the associated closed operator $\Lper:\domlper \longrightarrow L^2\Ito$, defined on a suitable domain reproducing the singularities and boundary conditions, has a purely discrete spectrum comprising conjugate pairs lying on the imaginary axis and accumulating only at $\pm i \infty$. This comes as a surprise at first sight, hinting that perhaps the dominant part of $\elle$ for fixed $\varepsilon$ is the antisymmetric term $u'(x)$. In reality, this is not a consequence of this suggestion, but rather a consequence of a delicate balance in obvious and hidden symmetries of the associated eigenvalue problem. Later it was shown \cite{2010BLM,2010BLM2} that this, and other remarkable spectral properties, also hold for more general $f(x)$.

Eigenvalue asymptotics for $\Lper$ were investigated in detail by Davies and Weir. For fixed $\varepsilon$, the leading order of the counting function is $2$, the same as for regular Sturm-Liouville problems. This rules out dominance of $u'(x)$ for fixed $\varepsilon$. It should be noted, however, that this term becomes loosely speaking ``dominant'' in the small  $\varepsilon$ regime. The  $n$th conjugate pair of eigenvalues of $\Lper$ converges to $\pm i n$ as $\varepsilon\to 0$. Deducing this latter property is far from straightforward, as the perturbation at $\varepsilon=0$ becomes singular.
 
Asymptotics of the counting function are closely linked with trace properties of the resolvent via Lidskii's theorem.
In the present paper we consider the same problem \eqref{fb_eq} but replacing $\sin(x)$ by a more general $f:\Rbb \longrightarrow \Rbb$ assuming that it is
\begin{enumerate}
\item \label{c1} absolutely continuous and $2\pi$-periodic, 
\item \label{c2} differentiable except possibly at a finite number of points excluding integer multiples of $\pi$,
\item \label{c3} $c=f'(0)\not=0$ exists and $f(x) = cx + O(x^{1+\delta})$ for some $\delta>0$ near $x=0$, and
\item \label{c4} $f(x+\pi) = -f(x)$, $f(-x) = -f(x)$ and $f(x) > 0$ for all $x \in \Ipl$.
\end{enumerate}
 In Theorem~\ref{schatten} we show that the resolvent of $\Lper$ lies in the $p$ Schatten-von Neumann class for all $p>2/3$. In Theorem~\ref{infinitely_many} we show that
it always has infinitely many eigenvalues. Both these results extend those of \cite{2010BLM2} and \cite[Proposition~4.3]{2009Chugunovaetal}.

The operator $i\Lper$ is not similar to a self-adjoint one in the case $f(x)=\sin(x)$. This is a direct consequence of the fact that the eigenfunctions and associated functions of $\Lper$ do not form an unconditional basis of $L^2\Ito$, \cite{2007Davies2, 2009Chugunovaetal}. In Theorem~\ref{non-unconditional-basis} below
we show that this also holds true for the more general $f$.

Basis properties of the eigenfunctions and associated functions of $\Lper$ are closely related with existence properties of solutions for the evolution problem \eqref{fb_eq}. As a forward-backward evolution problem, the regularity of these solutions is in itself unusual and hence worth examining. After setting a rigourous framework for solutions of \eqref{fb_eq}, we show in Theorem~\ref{singularity_I+} 
a non-existence result for any initial direction $g(x)$ which is not sufficiently regular. 
This property was examined in \cite{2009Chugunovaetal} for $f(x)=\sin(x)$.

\section{The inhomogeneous time independent equation} \label{section2}

We first establish a rigourous operator-theoretic framework for the differential expression $\elle[u]$. 

We will call a function $u:\Ito\longrightarrow \Cbb$ \emph{admissible} iff
\[
u \in \ACl\big(\Imi \cup \Ipl\big)\qquad \text{and} \qquad fu' + u \in \ACl\big(\Imi \cup \Ipl\big).
\]
Here and below
$\ACl(J)$ denotes the space of absolutely continuous functions 
on any open sub-interval of $J$. Let us examine the inhomogeneous problem
\be \label{inhom}
\elle[u] = F
\ee
for $F \in L^2\Ito$ and $u$ admissible. Consider the integrating factor equation
\be \label{intfact}
\frac{p'}{p} = \frac{f'}{f} + \frac{1}{\varepsilon f}.
\ee
Making the substitution $p = e^q$ gives $q' = \frac{f'}{f} + \frac{1}{\varepsilon f}$ and so
\beu
q(\pm x) = \int_{\pm\pi/2}^{\pm x} \left(\frac{f'(y)}{f(y)} + \frac{1}{\varepsilon f(y)}\right)dy + c_1^\pm \qquad \qquad \forall x \in\Ipl.
\eeu
Therefore,
\be \label{intfact-form}
p(\pm x) = c_2^\pm e^{\int_{\pm\pi/2}^{\pm x} \left(\frac{f'(y)}{f(y)} + \frac{1}{\varepsilon f(y)}\right)dy} \qquad \qquad \forall x \in \Ipl.
\ee
Thus $p$ is a non-vanishing function in $\ACl(\Imi \cup \Ipl)$. Multiplying by $p$ transforms \eqref{inhom} into another Sturm-Liouville equation in divergent form. Here we can actually pick any value of $c_2^\pm \in \Cc$, so the sign of $p$ can be fixed in $\Ipl$ and $\Imi$ separately.

\begin{lem} \label{sol-equiv}
Let $F \in L^1_\loc\Ito$. An admissible function $u$ satisfies \eqref{inhom} almost everywhere in $\Ito$, if and only if $u' \in \ACl(\Imi \cup \Ipl)$ and
\be \label{inhom-sl}
(pu')' = \frac{p}{\varepsilon f}F
\ee
almost everywhere in $\Ito$ where the non-vanishing function $p$ solves \eqref{intfact}.
\end{lem}
\begin{proof}
Without lost of generality we consider $\Ipl$ only, the same arguments applying to $\Imi$.

For the forward implication we suppose $u$ is admissible and $\elle[u] = F$. Since $u, fu' + u \in \ACl \Ipl$ and $f$ is non-vanishing on $\Ipl$, we see that $u' \in \ACl \Ipl$. Now
\besu
\frac{(pu')'}{p} & = \frac{p'u' + pu''}{p} = \left(\frac{f'}{f} + \frac{1}{\varepsilon f}\right)u' + u'' \\
& = \frac{1}{\varepsilon f}(\varepsilon f'u' + \varepsilon f'u'' + u'') = \frac{1}{\varepsilon f}(\varepsilon fu' + u)' = \frac{F}{\varepsilon f},
\eesu
which is \eqref{inhom-sl}.

For the reverse implication suppose $u' \in \ACl \Ipl$. Since $f \in \ACl \Ipl$ is positive we have that $fu',fu'+u \in \ACl \Ipl$. Knowing \eqref{inhom-sl} and rearranging the above calculation yields
\beu
\frac{1}{\varepsilon f}(\varepsilon fu' + u)' = \frac{(pu')'}{p} = \frac{p}{\varepsilon f}F.
\eeu
Therefore $\elle[u] = F$ almost everywhere.
\end{proof}

Formulation \eqref{inhom-sl} will prove useful in deriving the Green's function of $\Lper$.

\begin{lem}
The integrating factor $p$ in \eqref{intfact-form} satisfies
\be \label{asym-2}
p(x) \sim \left\{ \begin{array}{ll}
f(x)|x|^{1/(\varepsilon c)}, & \mbox{when} \quad x \to 0^\pm; \\
 f(x)|\pi \mp x|^{-1/(\varepsilon c)}, & \mbox{when} \quad x \to \pm\pi.
\end{array} \right.
\ee
\end{lem}

\begin{proof}
We compute
\beu
\frac{1}{\varepsilon f(y)} - \frac{1}{\varepsilon yf'(0)} = \frac{O(y^{1+\delta})}{\varepsilon yf'(0)(f'(0)y + O(y^{1+\delta}))} =: \eta(y),
\eeu
which is a Lebesgue integrable function in a neighbourhood of $y=0$. 
Consequently
\besu
\int_{\pi/2}^x \frac{1}{\varepsilon f(y)} dy & = \int_{\pi/2}^x \left(\frac{1}{\varepsilon yf'(0)} + \eta(y)\right) dy \\
& = \ln{x^{1/(\varepsilon f'(0))}} - \ln\left(\frac{\pi}{2}\right)^{1/(\varepsilon f'(0))} + \int_{\pi/2}^x \eta(y) dy.
\eesu
Thus, by \eqref{intfact-form},
\beu
p(x) = c_3^+ e^{\int_{\pi/2}^x \eta(y) dy} f(x)x^{1/(\varepsilon f'(0))}.
\eeu
This proves the result for $x \to 0^+$. The case $x \to 0^-$ is similar.

For the cases $x \to \pm\pi$ the argument is again similar, but the use of $f(y) = f'(0)y + O(y^{1+\delta})$ is replaced with $f(y) = -f'(0)(y-\pi) + O((y-\pi)^{1+\delta})$, which follows from the assumptions on $f$. 
\end{proof}

We set $w = p/f$. Then
\be \label{asym-1}
w(x) \sim \left\{ \begin{array}{ll}
 |x|^{1/(\varepsilon c)}, & \mbox{when} \quad x \to 0^\pm; \\
|\pi \pm x|^{-1/(\varepsilon c)}, & \mbox{when} \quad x \to \pm\pi.
\end{array} \right.
\ee
We point out two ``admissible'' solutions to the homogeneous problem
\be \label{hom}
\elle[u](x) = 0 \quad \mbox{for almost all} \quad x \in \Ito.
\ee
These are $\phi \equiv 1 \in L^2\Ito$ and $\psi = 1/w \not\in L^2\Ito$.
Note that the latter is only ensured by the choice $\varepsilon \leq 2/c$.
To see that $\psi$ is a solution, observe that our assumptions on $f$ yield $\psi' \in \ACl(\Imi\cup \Ipl)$ and
\be \label{calc}
p\left(\frac{f}{p}\right)' = \frac{f'p - fp'}{p} = f' - \frac{p'}{p}f
=  f' - \left(\frac{f'}{f} + \frac{1}{\varepsilon f}\right)f = \frac{1}{\varepsilon}.
\ee
Consequently, by Lemma \ref{sol-equiv}, $\psi$ satisfies \eqref{hom}.
If $f(x)=\sin(x)$, then $\psi(x)=|\cot(x/2)|^{1/\varepsilon}$.

We say that $u \colon J \to \Cbb$, is an \emph{admissible solution} of $\elle[u] = F$ if 
\begin{enumerate}
\item $u \in \ACl(J)$,
\item $fu' + u \in \ACl(J)$ and
\item $\elle[u](x) = F(x)$ for almost all $x \in J$.
\end{enumerate}

\begin{lem} Let $F \in L^1_{\rm loc}(\Imi \cup \Ipl)$. A function $u$ is an admissible solution to $\elle[u] = F$ in $\Imi \cup \Ipl$ if and only if, for some constants $k_1^\pm$ and $k_2^\pm$,
\be \label{rep}
u(x) = -\psi(x)\left(\int_{\pm\pi/2}^{x} \frac{F(y)}{\psi(y)} dy + k_2^\pm\right) + \left(\int_{\pm\pi/2}^{x} F(y) dy + k_1^\pm\right)
\ee
for almost every $x \in \Ipl$ (plus sign) or $x \in \Imi$ (minus sign).
\end{lem}

\begin{proof}
First we assume $u$ is given by \eqref{rep}. Then clearly $u \in \ACl(\Imi \cup \Ipl)$ and
\besu
u'(x) & = -\psi'(x)\left(\int_{\pm\pi/2}^x \frac{F(y)}{\psi(y)} dy + k_2^\pm\right) - \psi(x)\frac{F(x)}{\psi(x)} + F(x) \\
& = -\psi'(x)\left(\int_{\pm\pi/2}^x \frac{F(y)}{\psi(y)} dy + k_2^\pm\right).
\eesu
In addition, $u' \in \ACl(\Imi \cup \Ipl)$, so $fu' + u \in \ACl(\Imi \cup \Ipl)$, and, using \eqref{intfact},
\besu
\elle[u](x) & = (\varepsilon fu')'(x) + u'(x) \\
& = -\elle[\psi](x)\left(\int_{\pm\pi/2}^x \frac{F(y)}{\psi(y)} dy + k_2^\pm\right) + \varepsilon f(x) \psi'(x) \frac{F(x)}{\psi(x)} \\
& = \varepsilon f(x) \psi'(x) \frac{F(x)}{\psi(x)} = \frac{\varepsilon f(x)(p(x)f'(x) - p'(x)f(x))}{f(x)p(x)} \\
& = \varepsilon f(x)F(x)\left(\frac{f'(x)}{f(x)} - \frac{p'(x)}{p(x)}\right) = F(x),
\eesu
which proves the reverse implication of the lemma.

To prove the forward implication we first observe, by Lemma \ref{sol-equiv}, that an arbitrary admissible function solving \eqref{hom} also satisfies $(p w')' = 0$. Therefore, $p w'$ is constant so $w' = 1/p$. Integrating both sides and using \eqref{calc} gives us that
\be \label{hom-sol}
w = \alpha\psi + \beta
\ee
for some constants $\alpha,\beta \in \Cbb$.

Now suppose $\tilde{u}$ is an admissible solution, so $\elle[\tilde{u}] = F$ almost everywhere. Here we will consider the solution on $\Ipl$, the same argument giving the result in $\Imi$. Let $u$ be given by \eqref{rep} with $k_1^+ = k_2^+ = 0$. Then $\tilde{u} - u$ solves the homogeneous problem \eqref{hom} and so, by \eqref{hom-sol}, $\tilde{u} - u = \alpha \psi + \beta$ for some $\alpha,\beta \in \Cbb$. A rearrangement of this expression yields
\beu
u(x) = -\psi(x)\left(\int_{\pi/2}^x \frac{F(y)}{\psi(y)} dy + \alpha\right) + \left(\int_{\pi/2}^x F(y) dy - \beta\right),
\eeu
which is of the required form.
\end{proof}

Given $F \in L^2\Ito$, we wish to be able to solve \eqref{inhom} with $u \in \AC\Ito \cap L^2\Ito$ satisfying periodic boundary conditions at $\pm\pi$.

If we wish $u \in L^2\Ito$ in \eqref{rep}, then necessarily we require
\beu
k_2^\pm = -\int_{\pm\pi/2}^0 \frac{F(y)}{\psi(y)} dy.
\eeu
Indeed, note that as $\psi(x) \sim x^{-\frac{1}{c\varepsilon}}$ for $x \sim 0$ (from \eqref{asym-1}) we require that
\beu
\lim_{x \to 0} \left(\int_{\pm\pi/2}^x \frac{F(y)}{\psi(y)} dy + k_2^\pm\right) = 0.
\eeu
Therefore \eqref{rep} becomes
\be \label{rep-2}
u(x) = -\psi(x)\int_0^x \frac{F(y)}{\psi(y)} dy + \int_0^x F(y) dy + k^\pm,
\ee
for any $k^\pm \in \Cbb$ and $x \in \Imi\cup\Ipl$. Using the fact that $F \in L^2\Ito$ and \eqref{asym-1},
\besu
\left|\psi(x)\int_0^x \frac{F(y)}{\psi(y)} dy\right| & \leq |\psi(x)| \left(\int_0^\pi |F(y)|^2 dy \right)^{1/2} \left(\int_0^x |\psi(y)|^{-2} dy \right)^{1/2} \\
& \leq c_5|\psi(x)||x|^{\frac{1}{c\varepsilon} + \frac12} = O(x^{1/2})
\eesu
as $x \to 0$, and
\besu
\left|\psi(x)\int_0^x \frac{F(y)}{\psi(y)} dy\right| & \leq \left|\psi(x)\int_0^{\pm\pi/2} \frac{F(y)}{\psi(y)} dy\right| + \left|\psi(x)\int_{\pm\pi/2}^x \frac{F(y)}{\psi(y)} dy\right| \\
& \leq |\psi(x)| \left(c_6 + \left(\int_{\pm\pi/2}^{\pm\pi} |F(y)|^2 dy \right)^{1/2} \left(\int_{\pm\pi/2}^x |\psi(y)|^{-2} dy \right)^{1/2} \right) \\
& \leq c_7|\psi(x)||x\mp \pi|^{\frac12 - \frac{1}{c\varepsilon}} = O((x\mp\pi)^{1/2})
\eesu
as $x \to \pm\pi$, so $u$ of \eqref{rep-2} does indeed belong to $L^2\Ito$.

The requirement that $u \in \AC\Ito$ means that it must, in particular, be continuous. Therefore
\beu
k^+ = \lim_{x \searrow 0} u(x) = \lim_{x \nearrow 0} u(x) = k^-
\eeu
and so \eqref{rep} is further reduced to
\be \label{rep-3}
u(x) = \int_0^x \left(1 - \frac{\psi(x)}{\psi(y)}\right) F(y) dy + k,
\ee
for any $k \in \Cbb$ and $x \in I$.

Finally, the periodic boundary condition requires that both limits
\be
\lim_{x \to \pm\pi} u(x) = k + \int_0^{\pm\pi} F(y) dy
\ee
are equal. This is equivalent to
\beu
\int_{-\pi}^{+\pi} F(y) dy = 0
\eeu
and so the periodicity of $u$ is equivalent to the requirement that $F \perp 1$.

Therefore $u \in L^2\Ito$ is an admissible solution to $\elle[u] = F$ for some $F \in L^2\Ito$ and satisfies the periodic boundary condition, if and only if it has the form \eqref{rep-3} for some $k \in \Cbb$ and $F \in L^2\Ito \ominus \spa\{1\}$. We now describe the operator theoretical setting for the differential expression
$\elle$. Let
\beu
\Dmax = \{u \colon \Ito \to \Cbb \, | \, \mbox{\eqref{rep-3} holds for some $k \in \Cbb$ and $F \in L^2\Ito$} \}.
\eeu
By the above argument we know that
\beu
\Dmax = \{u \in L^2\Ito \, | \, u\mbox{ is an admissible function and $\elle[u] \in L^2\Ito$}\},
\eeu
that is, it is the maximal domain associated with the differential operator $\elle$. We define
\besu
\DomL & = \{u \in \Dmax \, | \, \ell_\varepsilon[u] \perp 1 \} \\
& = \{u \in \Dmax \, | \, \lim_{x \to \pi} u(x) = \lim_{x \to -\pi} u(x) \}
\eesu
and denote by $\Lper$ the differential operator $(\elle,\DomL)$.
\begin{lem}\label{closed}
The operator $\Lper:\DomL\longrightarrow L^2\Ito$ is closed.
\end{lem}
\begin{proof}
We take a sequence $\{u_n\}_{n=1}^\infty \subset \DomL$ such that $u_n \to u$ in $L^2\Ito$ and $\elle[u_n] = F_n$ ($n \in \Nbb$) are such that $F_n \to F \in L^2$ with $F_n \in L^2\Ito \ominus \spa\{1\}$ for all $n \in \Nbb$. We need to prove $u \in \DomL$ and $\elle[u] = F$. From \eqref{rep-3} we know that
\beu
u_n = T(F_n) + k_n,
\eeu
where $T(F)(x) = \int_0^x \left(1 - \psi(x)/\psi(y)\right) F(y) dy$. The operator $T$ is bounded on $L^2\Ito$ as its kernel is bounded (as can be seen via \eqref{asym-1}). Consequently, as $F_n \to F$, we have that $k_n \to k$ for some $k\in \Cbb$. Thus $u = T(F) + k$ and so $u \in \Dmax$ and $\elle[u] = F$. Moreover, since $F_n \perp 1$, $F \perp 1$ and so $u \in \DomL$.
\end{proof}


\section{Trace properties of the resolvent of $\Lper$}

For $F\in L^2\Ito$, let
\beu
TF(x)=T[F](x) = \int_0^x \left(1 - \frac{\psi(x)}{\psi(y)}\right) F(y) dy = \int_{-\pi}^\pi G(x,y) F(y) dy
\eeu
where
\be \label{green_ker}
G(x,y) = \left\{\begin{array}{ll}
\left(1 - \psi(x)/\psi(y)\right), & \mbox{if $0 < y < x$;} \\
-\left(1 - \psi(x)/\psi(y)\right), & \mbox{if $x < y < 0$;} \\
0, & \mbox{otherwise.}
\end{array}\right.
\ee
See Figure~\ref{fig1}. By virtue of \eqref{asym-1}, $G$ is bounded and so $T$ is a Hilbert-Schmidt operator.

\begin{figure}[hth]
\centerline{\includegraphics[height=9cm, angle=0]{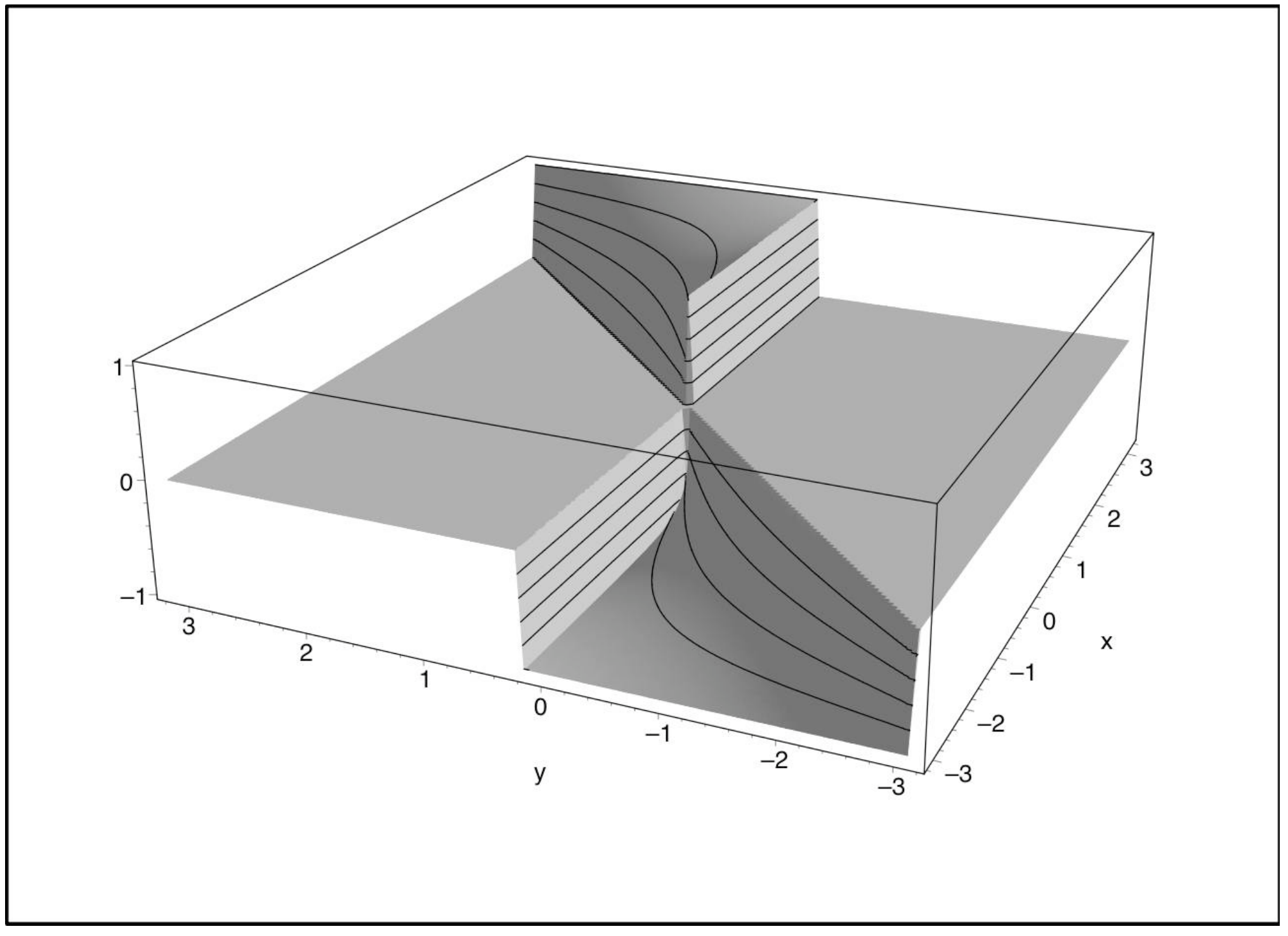}}   
\caption{Green's function $G(x,y)$ for $f(x)=\sin(x)$ and $\varepsilon=1$. \label{fig1}}
\end{figure}

\begin{lem} \label{five}
Let $u \in \DomL$. Then $u \perp 1$ and $\Lper u = F$, if and only if
\beu
u(x) = TF(x) - \frac{1}{2\pi}\langle TF,1\rangle = (I - \frac{1}{2\pi}|1\rangle \langle 1|)TF(x) \qquad \text{and}\qquad F\perp 1.
\eeu
\end{lem}

\begin{proof}
Suppose that $u \perp 1$ and $\Lper u = F$. Therefore, by \eqref{rep-3},
\beu
0 = \int_{-\pi}^\pi u(x) dx = \int_{-\pi}^\pi TF(x) + k dx = \langle TF,1\rangle + 2\pi k
\eeu
and so $2\pi k = -\langle TF,1\rangle$.
This proves one direction of the implication. The other direction is trivial.
\end{proof}

Let $\tilde{D} = \DomL \ominus \spa\{1\}$. Then $\Lper$ has the block diagonal
representation
\beu
\Lper = \Lper \upharpoonright_{\tilde{D}} \oplus 0 \colon \tilde{D} \oplus \spa\{1\} \to \spa\{1\}^\perp \oplus \spa\{1\}.
\eeu
Let
\beu
\tilde{T}F(x) = TF(x) - \frac{1}{2\pi}\langle TF,1\rangle
\eeu
for $F \in L^2\Ito \ominus \spa\{1\}$ so that, by Lemma \ref{five}, $\tilde{T} \colon \spa\{1\}^\perp \to \tilde{D}$. Then $(\Lper \upharpoonright_{\tilde{D}})^{-1} = \tilde{T}$ and, in particular, $0 \not\in \spec(\Lper \upharpoonright_{\tilde{D}})$. Since $\tilde{T}$ is a rank-$1$ perturbation of $T\upharpoonright_{\spa\{1\}^\perp}$ and the generalised singular value decomposition preserves any block structure of operators, we know the following.
\begin{lem}
The resolvent of $\Lper$ is in the $p$-Schatten class $\sch_p$ if and only if $T \in \sch_p$.
\end{lem}

Given an $r > 0$, we will denote
\beu
\sch_{p>r} = \bigcap_{p > r} \sch_p.
\eeu
In order to find the Schatten properties of $T$ we consider below a generic lemma which, keeping the notation tidy, we formulate in $L^2(0,\pi)$. 
It can be easily seen that the interval $(0,\pi)$ can be replaced with any other bounded interval.
\begin{figure}[hth]
\centerline{\includegraphics[height=6cm, angle=0]{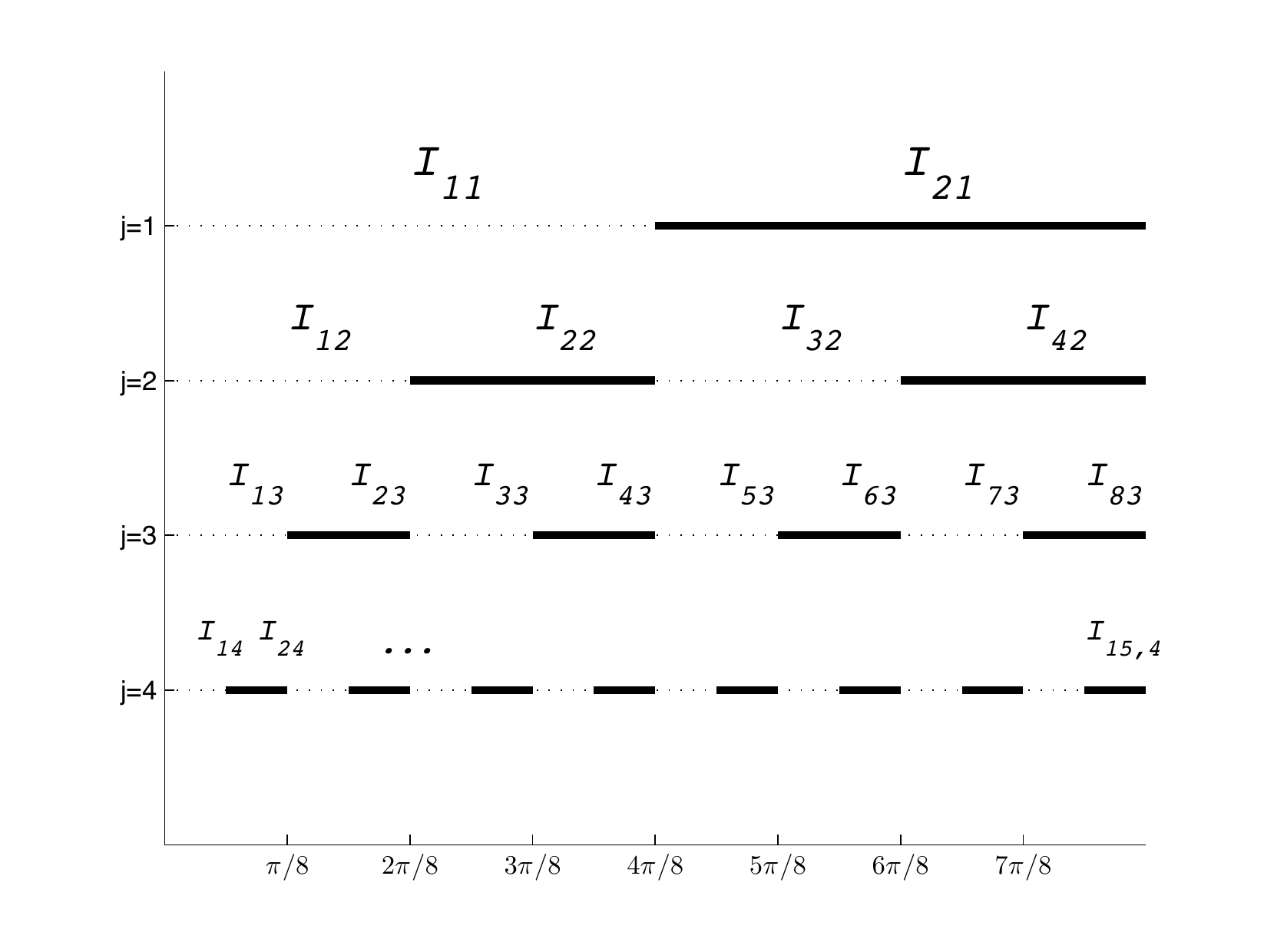}
\includegraphics[height=6cm, angle=0]{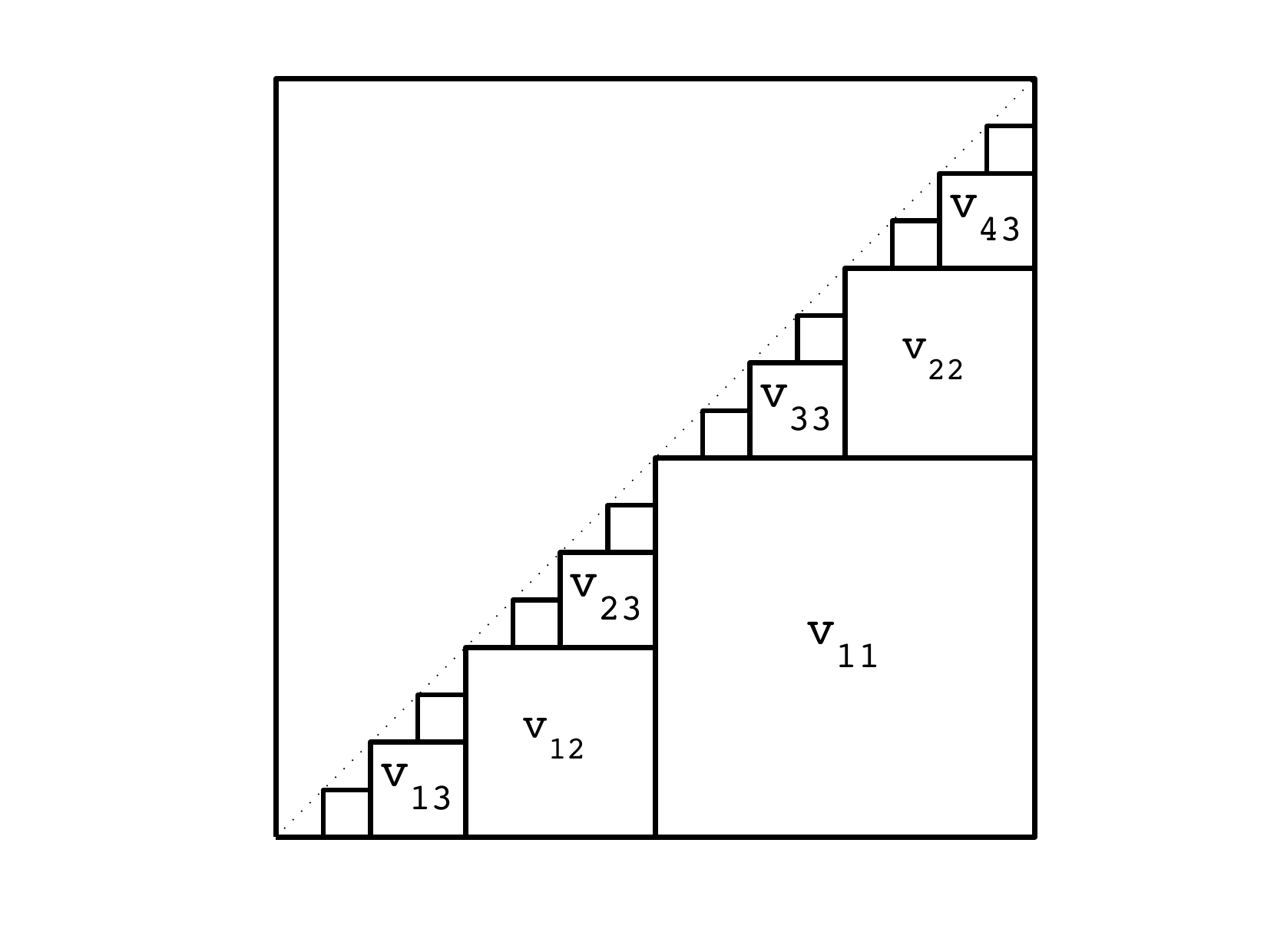}}
\caption{Left: For each fixed $j$ the intervals $I_{2k,j}$ and $I_{2k-1,j}$ are 
disjoint. Right: Illustration of the support of $v_{kj}(x,y)$. \label{fig2}}
\end{figure}

\begin{lem} \label{ablem}
Let $a,b \colon (0,\pi) \to \Cbb$ be two continuous functions and denote by $v(x,y) = a(x)b(y)$. Let
\beu
S(u)(x) = \int_0^x v(x,y) u(y) dy.
\eeu
If
\beu
v(x,y) = \left\{ \begin{array}{ll}
O(x^{-\alpha}), & \mbox{when} \quad (x,y) \to (0,0); \\
O((\pi - y)^{-\beta}), & \mbox{when} \quad (x,y) \to (\pi,\pi),
\end{array} \right.
\eeu
for some $0 \leq \alpha,\beta < 1/2$, then $S \in \sch_{p>r(\alpha,\beta)}$ where
\beu
r(\alpha,\beta) = \max\{1/(1-\alpha),1/(1-\beta)\}.
\eeu
\end{lem}

\begin{proof}
Let $\chi_{ij} = \chi_{I_{ij}}$ be the characteristic function of the interval $I_{ij} = [\frac{(i-1)\pi}{2^{j}},\frac{i\pi}{2^{j}})$ for $i = 1,\dots,2^j$ and $j \in\Nbb$. See Figure~\ref{fig2} (left). By construction, we have the pointwise equality
\beu
v(x,y) \chi_\Omega(x,y) = \sum_{j=1}^\infty v_j(x,y),
\eeu
where 
\beu
v_j(x,y) = \sum_{k=1}^{2^{j-1}} a(x)\chi_{2k,j}(x) b(y) \chi_{2k-1,j}(y) = \sum_{k=1}^{2^{j-1}} v_{kj}(x,y)
\eeu 
and $\chi_\Omega$ is the characteristic function of 
$\Omega = \{(x,y) \, | \, 0 < y < x < \pi \}$. See Figure~\ref{fig2} (right).

Let $S_j = \sum_{k=1}^{2^{j-1}} S_{kj}$ where
\beu
S_{kj}u(x) = |b \chi_{2k-1,j}\rangle\langle a \chi_{2k,j}|u(x) = \int_0^\pi v_{kj}(x,y)u(y) dy.
\eeu
Since, for each fixed $j$, all the intervals indexed by $k$ are disjoint, this is a singular value decomposition for $S_j$. The $p$th Schatten class norm of $S_{kj}$ is
\beu
\alpha_{kj} := \|S_{kj}\|_p = \|a\|_{L^2(I_{2k,j})}\|b\|_{L^2(I_{2k-1,j})}
\eeu
for all $1 \leq p \leq \infty$ (in particular, it is independent of $p$). Thus, if we prove that
\be \label{decom-0}
\sum_{j=1}^\infty \|S_j\|_p = \sum_{j=1}^\infty \left( \sum_{k=1}^{2^{j-1}} |\alpha_{kj}|^p\right)^\frac{1}{p} < \infty,
\ee
we would have that $S \in \sch_p$ as the sum $S = \sum_{j=1}^\infty S_j$ would converge in the $p$th Schatten norm $\| \cdot \|_p$.

For $j \geq 2$ we can certainly separate the sum
\beu
S_j = \sum_{k=1}^{2^{j-2}} S_{kj} + \sum_{k=2^{j-2}}^{2^{j-1}} S_{kj} = S_j^\ell + S_j^u
\eeu
and
\beu
S = S_{11} + \sum_{j=2}^\infty S_j^\ell + \sum_{j=2}^\infty S_j^u = S_{11} + S^\ell + S^u,
\eeu
so the proof will be complete (as was the case with \eqref{decom-0}) if we can show
\begin{align}
\sum_{j=2}^\infty \|S_j^\ell\|_p & < \infty \quad \mbox{for} \quad p > \frac{1}{1-\alpha} \quad \mbox{and} \label{decom-1} \\
\sum_{j=2}^\infty \|S_j^u\|_p & < \infty \quad \mbox{for} \quad p > \frac{1}{1-\beta}. \label{decom-2}
\end{align}
We will just prove \eqref{decom-1}, the proof of \eqref{decom-2} being similar.

The hypotheses of the lemma guarantee that there exists a constant $c_8 > 0$ such that 
\beu
|v(x,y)| \leq c_8 x^{-\alpha} \quad \mbox{for all} \quad (x,y) \in \bigcup_{\tiny \begin{array}{c} j \geq 2 \\ 1 \leq k \leq 2^{j-2} \end{array}} I_{2k,j} \times I_{2k-1,j}.
\eeu
Then, for $1 \leq k \leq 2^{j-2}$,
\bes \label{eight}
\alpha_{kj}^2 & = \int_{I_{2k,j}} \int_{I_{2k-1,j}} |v(x,y)|^2 dx dy \\
& \leq c_8 \int_{I_{2k,j}} \int_{I_{2k-1,j}} x^{-2\alpha} dx dy = c_8 \pi 2^{-j} \int_{\pi(2k-1)2^{-j}}^{\pi(2k)2^{-j}} x^{-2\alpha} dx \\
& = c_9 2^{-j} \left[\left(\frac{2k}{2^{j}}\right)^{1-2\alpha} - \left(\frac{2k-1}{2^{j}}\right)^{1-2\alpha}\right].
\ees
Letting $N = 1/(1-2\alpha)$, so that $0 \leq \alpha < 1/2$ if and only if $N \geq 1$. We have that
\beu
1-r \leq \frac{1-r^N}{1+r^{N+1}} \quad \mbox{for all $0 \leq r \leq 1$ and $N \geq 1$}.
\eeu
Thus the right-hand side of \eqref{eight} becomes
\besu
& c_92^{-j}\left(\frac{2k}{2^j}\right)^\frac{1}{N}\left[ 1 - \left(\frac{2k-1}{2k}\right)^\frac{1}{N}\right]
\leq c_92^{-j(1 + 1/N)}\left(2k\right)^\frac{1}{N}\left[\frac{1 - \frac{2k-1}{2k}}{1 + \left(\frac{2k-1}{2k}\right)^\frac{N-1}{N}}\right] \\
& = c_92^{-j(1 + 1/N)}\left(2k\right)^\frac{1}{N}\left[\frac{1}{\left(2k\right)^\frac{N-1}{N} + \left(2k-1\right)^\frac{N-1}{N}}\right]
\leq c_92^{-j(1 + 1/N)}.
\eesu
Therefore $\alpha_{kj} \leq c_9 2^{-j(1 + 1/N)/2}$ for $1 \leq k \leq 2^{j-2}$. Thus
\beu
\|S_j^\ell \|_p = \left(\sum_{k=1}^{2^{j-2}} \alpha_{kj}^p \right)^\frac{1}{p} \leq c_9\left(2^{j-2}2^{-jp(1 + 1/N)/2}\right)^\frac{1}{p} = c_9 (2^\frac{1}{p})^{j(1 - p(1 + 1/N)/2)}.
\eeu
And so, $\sum_{j=1}^\infty \|S_j^\ell \|_p$ converges if $1 - p(1 + 1/N)/2 < 0$, which is equivalent to $p > 1/(1-\alpha)$.
\end{proof}

By virtue of this lemma we are able to prove the following theorem.

\begin{thm} \label{schatten}
$T \in \sch_{p>2/3}$.
\end{thm}

\begin{proof}
Let $\Omega^{\pm}=\{(x,y)\in\Rbb^2:0<\pm y<\pm x<\pm \pi\}$. Let
\beu
G^{\pm}(x,y)=\pm \left(1-\frac{\psi(x)}{\psi(y)}\right)\chi_{\Omega^{\pm}}(x,y).
\eeu
Then $T=T^+ +T^-$ where 
\beu
    T^{\pm}u(x)=\int_{-\pi}^\pi G^{\pm}(x,y) u(y) dy.
\eeu
The proof reduces to showing that $T^\pm \in \sch_{p>2/3}$. We shall only 
give the details for the case of $T^+$, the other case being analogous.

Set $T^+=T^+_1+T^+_2+T^+_3$ where
\begin{gather*}
T^+_{j}u(x)  = \int_{-\pi}^{\pi} G_j^+(x,y)u(y) dy \qquad
\text{for} \qquad G_j^+(x,y)= G^+(x,y)\chi_{\Omega_j^+}(x,y), \\
\Omega_1^+=\{0<x<y<\pi/2\}, \qquad
\Omega_2^+=[\pi/2,\pi]\times[0,\pi/2], \\
\text{and} \qquad \Omega_1^+=\{\pi/2<x<y<\pi\}.
\end{gather*}
We prove that each of the $T^+_j \in \sch_{p>2/3}$.

Note that $T^+_2\in\sch_{p>0}$ as it is of rank two. Let us show that $T^+_1 \in \sch_{p>2/3}$. Let
\beu
K_1^+(x,y) = \partial_xG(x,y)\chi_{\Omega_1^+}(x,y) = 
\left\{ \begin{array}{ll}
\frac{\psi'(x)}{\psi(y)}, & \mbox{when} \quad (x,y) \in \Omega_1^+; \\
0, & \mathrm{otherwise}.
\end{array} \right.
\eeu
Then, for any $0 < \alpha < 1/2$,
\besu
T_{1}^+u(x) & = \int_{0}^{\pi/2} \int_0^x K_1^+(z,y)u(y) dz dy \\
& = \int_0^x z^{-\alpha} \int_{0}^{\pi/2} z^\alpha K_1^+(z,y)u(y) dy dz = S_{1\alpha}^+(
R_{1\alpha}^+[u])(x).
\eesu
By virtue of Lemma \ref{ablem}, $S_{1\alpha}^+ \in \sch_{p > 1/(1-\alpha)}$. On the other hand, by \eqref{asym-2}, \eqref{asym-1} and \eqref{calc}, $|z^\alpha K_1^+(z,y)| \sim z^{\alpha - 1}$ for $(z,y) \sim (0,0)$, so
\beu
\int_0^\pi \int_0^{\pi/2} |z^\alpha K_1^+(z,y)|^2 dy dz < \infty
\eeu
whenever $\alpha > 0$. Thus $R_{1\alpha}^+ \in \sch_2$. As
$\alpha$ can be taken arbitrarily close to zero, $\|T_1^+\|_p\leq 2^{1/p}
\|S^+_{1\alpha}\|_q \|R^+_{1\alpha}\|_2$ with $1/p=1/q+1/2$ and $q>1$
but arbitrarily close to one. Therefore $T^+_1 \in \sch_{p>2/3}$.

Arguing similarly for $T^+_{3}$ we set
\beu
K^+_3(x,y) = \partial_yG(x,y)\chi_{\Omega_3^+}(x,y) = \left\{ \begin{array}{ll}
-\frac{\psi(x)\psi'(y)}{\psi(y)^2}, & \mbox{when} \quad (x,y) \in \Omega^+_3; \\
0, & \text{otherwise}.
\end{array} \right.
\eeu
Then, for any $0 < \beta < 1/2$,
\besu
T^+_{3}u(x) & =  -\int_{\pi/2}^{\pi} \int_y^\pi 
K_3^+(x,z)u(y) dz dy \\
& = -\int_{\pi/2}^{\pi} K_3^+(x,z)(\pi - z)^\beta \left(\int_{\pi/2}^z (\pi - z)^{-\beta} u(y) dy\right) dz \\
& = R_{3\beta}^+(S_{3\beta}^+[u])(x).
\eesu
Once again, by Lemma \ref{ablem}, $S_{3\beta}^+ \in \sch_{p > 1/(1-\beta)}$ 
and $|(\pi - z)^\beta K_3^+(x,z)| \sim (\pi - z)^{\beta - 1}$ 
for $(x,z) \sim (\pi,\pi)$, so
\beu
\int_0^\pi \int_{\pi/2}^{\pi} |K_3^+(x,z)(\pi - z)^\beta|^2 dz dx < \infty
\eeu
whenever $\beta > 0$. Thus $R_{3\beta}^+ \in \sch_2$. 
As $\beta$ can be taken arbitrarily close to zero, the proof follows in similar
fashion as the previous case.\end{proof}


\section{The forward-backward heat equation} \label{evol_prob}

By virtue of \cite[Theorem~3.3]{2010BLM}, the spectrum of $\Lper$ is contained in the purely imaginary axis. 
If $i \Lper$ was similar to a self-adjoint operator, it would be the generator of a unitary one-parameter
semigroup. In \cite{2009Daviesetal} it was shown that the latter is actually not the case for $f(x)=\sin(x)$. In this context we now examine more closely the evolution equation associated to $\Lper$ following the ideas of \cite{2009Chugunovaetal}.  

Let $T>0$. Consider the evolution problem
\begin{equation} \tag{B} \label{B}
 \left\{\begin{aligned}
 &\partial_t u(t,x)+\Lper u(t,x)=0 & \text{a.e. } (t,x)\in (0,T)\times \Ito \\
 &u(t,\cdot)\in \DomL & \forall t\in (0,T) \\
 &u(0,x) =g(x)  & \text{a.e. } x\in\Ito.
 \end{aligned} \right.
\end{equation}
We wish to define in a precise manner the notion of a solution of \eqref{B}. If $\Lper \phi=i\lambda \phi$ for $\lambda\in \Rr$ and $g(x)=\phi(x)$,
then formally we have a global solution of \eqref{B} given by $u(t,x)=e^{-i\lambda t} \phi(x)$
for any $T>0$. In an analogous fashion, we can generate solutions which are global whenever
$g$ is a finite linear combination of eigenfunctions of $\Lper$.  
      
We will denote the space of \emph{admissible trajectories} in which the solutions of
\eqref{B} lie by
\[
    \CT_T:=W^{1,2}_{2,\loc}\big((0,T)\times [\Imi\cup\Ipl]\big)\cap C\left([0,T); L^2\Ito\right).
\]

Here and below we follow closely the notation of \cite{2008Krylov}, where $W^{r,k}_2(\Omega)$ is the \emph{parabolic Sobolev space} for an open region $\Omega\subset \Rr^2$ of function with derivatives in the $x$-variable up to order $k$ and derivatives in the $t$-variable up to order $r$. The space $W^{r,k}_{2,\loc}(\Omega)$ is defined by saying $v\in W^{r,k}_{2,\loc}(\Omega)$ iff $\phi v\in  W^{r,k}_2(\Rr^2)$ for all smooth cut-off functions $\phi$ whose support is inside a compact proper subset of $\Omega$. Note that if $v \in \CT_T$,
then $\partial_tv+\Lper v\in L^2((0,T)  \times \Ito)$. Here and below, expressions involving partial derivatives
will always mean partial derivatives in the distributional (Sobolev) sense.  
 
By a \emph{solution} of \eqref{B} we mean $u\in \CT_T$ satisfying all conditions
in \eqref{B} for some $T>0$.

To simplify notation, we will denote functions and their corresponding restrictions to subdomains (or extensions
to larger domains) with the same letter. Extensions to larger domains are
assumed to be ``up to a set of measure zero''. 

In the following lemma it is crucial that $I_0\subset \Ipl$, however note that
there are no restrictions on the position of $J$ relative to $t=0$.

\begin{lem} \label{auxiliary2}
Let $f\in C^k(\Rr)$ for some $k\in \Nn$. Let $I_0\subset \overline{I}_0\subset \Ipl$ and
$J\subset \Rr$ be two open intervals. If $v\in W^{1,2}_{2,\loc}(J\times I_0)$ is
such that $\partial_tv+\elle v=0$, then $v\in W^{1,2+k}_{2}(V)$ for any $V\subset J\times I_0$.
\end{lem}
\begin{proof}The proof follows closely that of \cite[Corollary~2.4.1]{2008Krylov}.
Let $V_0,\,V_1$ be two open sets, such that $V_0\subset \overline{V_0} \subset
V_1\subset \overline{V_1} \subset J\times I_0$. Consider a cutoff
function associated to $V_0$ and $V_1$: $\zeta\in C^\infty_c(\Rr^2)$,
such that $\zeta(t,x)=1$ for $(t,x)\in V_0$ and  $\zeta(t,x)=0$ for $(t,x)\in \Rr^2\setminus V_1$. A straightforward calculation yields
\[
     (\partial_t+\elle)(\zeta v)=[v(\partial_t+\elle)\zeta]+[2\varepsilon f\zeta'v'].
\]
Elementary properties of the parabolic Sobolev spaces ensure that the first summing term on the right lies in $W^{1,2}_2(\Rr^2)$ and the
second one lies in $W^{1,1}(\Rr^2)$. Note that here $k\geq 1$ is required.
Thus the whole expression lies in $W^{0,1}_2(\Rr^{2})$. 

By virtue of \cite[Corollary~2.3.3]{2008Krylov}, $\zeta v\in W^{1,3}_2(\Rr^2)$.
As $\zeta=1$ in $V_0$, the parabolic Sobolev embedding theorem ensures
that $u\in W^{1,3}(V_0)$. This shows the lemma for $k=1$. If $k>1$, on the other hand,
the argument can be repeated until we get $v\in W^{1,2+k}_2(V)$.
\end{proof}

We now combine Lemma~\ref{auxiliary2} with \cite[Theorem~2.2.6]{2008Krylov}, in order to get the 
following non-existence result for \eqref{B}.

\begin{thm} \label{singularity_I+}
Let $f\in C^k(\Rr)$ for some $k\in \Nn$. Let $g\in L^2\Ito$. If there exists
a solution $u\in \CT_\delta$ of \eqref{B} for some
$\delta>0$, then $g\in C^\kappa(I_0)$ for any $I_0\subset \overline{I_0}\subset \Ipl$
where
\[
        \kappa =\left\{ \begin{aligned} &\frac{k-1}{2} & k\ \text{odd} \\
             & \frac k2 & k\ \text{even}. 
             \end{aligned} \right.
\]
\end{thm}
\begin{proof}
Let $J=(-\delta,\delta)$. The idea of the proof is to ``glue'' together
$u$ with a solution $u^d$ of a Dirichlet evolution problem associated to
$\Lper$ in $(-\delta,0)\times I_0$ and observe that the ``seam'' of this gluing
is $g(x)$. 

Let $u^d(t,x)$ for $(t,x)\in (0,\delta)\times I_0$ be a solution of the parabolic problem
\begin{equation} \label{dirichlet}
 \left\{\begin{aligned}
 &\partial_t u^d-\Lper u^d=0 & \forall (t,x)\in (0,\delta)\times I_0 \\
 &u^d(t,\min I_0)=u^d(t,\max I_0)=0 & \forall t\in (0,\delta) \\
 &u^d(0,x) =g(x)  & \forall x\in I_0.
 \end{aligned} \right.
\end{equation}
Since $f$ is positive definite in $I_0$, \eqref{dirichlet} has a unique (actually global in time) solution. By virtue of Lemma~\ref{auxiliary2}, 
\begin{equation} \label{regularity}
u^d,u\in W^{1,2+k}_2((0,\delta)\times I_0).
\end{equation}
 Let
\[
  v(t,x) =\left\{ \begin{aligned} &u^d(-t,x) & -\delta<t\leq 0,\, x\in I_0 \\
             & u(t,x) & 0\leq t <\delta,\, x\in I_0. 
             \end{aligned} \right.  
\]
Note that the change in the sign for $t$ ensures that in the region
$(-\delta,0)\times I_0$, $\partial v_t+\Lper v=0$.
By construction $v\in L^2(J\times I_0)$. We show that $v\in W^{1,2}_2(J\times I_0)$.

Let us verify firstly that $\partial_tv\in L^2(J\times I_0)$. Let
\[
  h(t,x) =\left\{ \begin{aligned} &-\partial_t u^d(-t,x) & -\delta<t< 0,\, x\in I_0 \\
             & \partial_t u(t,x) & 0< t <\delta,\, x\in I_0. 
             \end{aligned} \right.  
\]
By construction $h\in L^2(J\times I_0)$. We now show that $\partial_t v=h$.
For almost every $x\in I_0$,
\[
v(\cdot,x)\in C(J)\cap \left[ W^1_2(-\delta,0)\cup
W^1_2(0,\delta)\right].
\]
Then $\partial_tv(\cdot,x)=h(\cdot,x)\in L^2(J)$ for almost every $x\in I_0$. 
Thus
\[
    \int_J h(t,x)\phi(t) dt =-\int_J v(t,x) \partial_t\phi(t) dt
    \qquad \forall \phi\in C^{\infty}_c(J)
\] 
for almost every $x\in I_0$. Hence
\[
    \int_{I_0}\int_J h(t,x)\phi(t,x) dt \,dx =-\int_{I_0}\int_J v(t,x) \partial_t\phi(t,x) dt \,dx
    \qquad \forall \phi\in C^{\infty}_c(J\times I_0),
\]
so that $\partial_tv=h\in L^2(J\times I_0)$.

Now let us prove that $\partial_xv\in L^2(J\times I_0)$. Let
\[
  k(t,x) =\left\{ \begin{aligned} &\partial_x u^d(-t,x) & -\delta<t< 0,\, x\in I_0 \\
             & \partial_x u(t,x) & 0< t <\delta,\, x\in I_0. 
             \end{aligned} \right.  
\]
Then $k\in L^2(J\times I_0)$. Let us show that $\partial_xu=k$. Note that
$u^d(-t,\cdot),u(t,\cdot)\in C^2(I_0)$ for all $t\in(0,\delta)$ as
a consequence of \eqref{regularity}.  Hence,
\[
    \int_{I_0} k(t,x)\phi(x) dx =-\int_{I_0} v(t,x) \partial_x\phi(x) dx
    \qquad \forall \phi\in C^{\infty}_c(I_0)
\]
for all $t\in J\setminus \{0\}$. Then 
\[
    \int_{I_0}\int_J k(t,x)\phi(t,x) dt \,dx =-\int_{I_0}\int_J v(t,x) \partial_x\phi(t,x) dt \,dx
    \qquad \forall \phi\in C^{\infty}_c(J\times I_0).
\]
Hence $\partial_xv=k\in L^2(J\times I_0)$. 

The proof that $\partial_{xx}v\in  L^2(J\times I_0)$ is similar. Note that here is crucial that $k\geq 1$. Hence $v\in W^{1,2}_2(J\times I_0)$ as needed.

We complete the proof of the theorem as follows. By Lemma~\ref{auxiliary2},
$v\in W^{1,2+k}_2(J\times I_0)$. Since $g(x)=v(0,x)$, taking $r=1$ and
the ``$k$'' of the theorem as ``$2+k$'' in \cite[Theorem~2.2.6]{2008Krylov},
we get for $\kappa<\frac{1+k}{2}$ (which is equivalent to $\frac{2\kappa+1}{2(2+k)}+\frac12<1$ in the mentioned theorem) that $g\in C^{\kappa}(I_0)$.
\end{proof}

One relevant question in the context of this theorem is the existence of solution of \eqref{B} (for $T$ sufficiently small), if $\mathrm{supp} (g)\subset J\subset \overline{J} \subset \Imi$. A positive answer would certainly be of interest. In this respect the ``barrier'' at $x=0$ may prevent this solution 
propagating from $\Imi$ to $\Ipl$ and make it lose its regularity.


\section{Basis properties of the eigenfunctions of $\Lper$}

With Theorem~\ref{singularity_I+} at hand and the ``diagonalisation'' lemma below, we can establish properties of the set of eigenfunctions of $\Lper$. First, however, we deal with one piece of unfinished business from \cite{2010BLM}. A proof of the second statement in the following theorem by different methods can be found in \cite[Proposition~5.5]{2009Chugunovaetal} for the
case $f(x)=\sin(x)$.

\begin{thm} \label{infinitely_many} 
Suppose that $f$ satisfies conditions \ref{c1}-\ref{c4} from Section~\ref{section1}. The operator $\Lper$ has infinitely many eigenvalues. Moreover, all these eigenvalues are algebraically simple. Consequently, the eigenspaces of $\Lper$ contain only eigenfunctions
and no associated functions.
\end{thm}
\begin{proof} Following the notation in \cite{2010BLM}, we introduce the meromorphic function
\begin{equation}    \label{pro_rho}
\rho(z) = \frac{\phi(\pi;-iz^2)}{\phi(\pi;iz^2)} = \frac{\Pi_{n=1}^\infty\left(1-\frac{z^2}{\alpha_n^2}\right) }{ \Pi_{n=1}^\infty\left(1+\frac{z^2}{\alpha_n^2}\right) }= \prod_{n=1}^\infty  \frac{1-\frac{z^2}{\alpha_n^2}}{1+\frac{z^2}{\alpha_n^2}}.
\end{equation}
Here $\phi(x;\lambda)$ is the unique solution of the differential equation $i\ell_\epsilon \phi =\lambda \phi$ satisfying $\phi(0;\lambda) = 1$. The sequence $(\alpha_n)$ consists entirely of positive numbers and its asymptotic behaviour is $\alpha_n=O(n)$ as $n\to\infty$, so it grows sufficiently rapidly to ensure the convergence of the products on the right hand side. It was proved in \cite{2010BLM} that $\lambda$ is an eigenvalue of $iL_\epsilon$ if and only if $\lambda = -iz^2$ for $z\in\Cc$ such that $\rho(z)=1$.
Moreover it was shown that $|\rho(z)|=1$ on, and only on, the lines $\arg(z) \equiv \pi/4$ (modulo $\pi/2$). 

For the first part of the theorem it suffices to show that there
exist infinitely many $r\in\Rr$ such that $\rho(r\exp(i\pi/4)) = 1$. 
To this end we put $\rho(r\exp(i\pi/4)) = \mu(r)=\exp(i\Theta(r))$ which maps $\Rr$ smoothly into the unit circle, $\mathbb{T}$. Here the function $\Theta(r)$ is chosen by picking a branch of the logarithm in the expression
\begin{align*} 
i\Theta(r) &= \log(\mu(r))=\sum_{n=1}^\infty\left\{ \log\left(1-i\frac{r^2}{\alpha_n^2}\right) - \log\left(1+i\frac{r^2}{\alpha_n^2}\right)\right\}  \\
&= 2i\sum_{n=1}^\infty \arg \left(1-i\frac{r^2}{\alpha_n^2}\right) . 
\end{align*}
For convenience we choose the branch $[-\pi,\pi)$ and observe that all summations in the above converge as a consequence of the convergence of \eqref{pro_rho}. We now show that $\rho$ passes through any point of $\mathbb{T}$ (including the needed value 1) an infinite number of times as $|r|$ increases. 

The meromorphic function $\rho(1/z)$ has an essential singularity at the origin.
Therefore, in any neighbourhood of $0$, $\rho(1/z)$ assumes each complex number, with the possible exception of one, infinitely many times.
Since $|\rho(z)|=1$ only on the rays $\arg(z) \equiv \pi/4$ (modulo $\pi/2$) and
$\rho(\overline{z})=\overline{\rho(z)}$, necessarily 
$\mu(r)$ assumes any value of $\mathbb{T}$, with the possible exception of one, infinitely many times for $r\in \Rr$. Now
\[
     \partial_r \mu(r)=i\Theta'(r) \exp\left(i \Theta(r)\right).
\]
Here $\Theta'(r)$ is defined as the derivative from the ``right'', if $r$ is such that $\Theta(r)=-\pi$.
Differentiation yields 
\[ 
\Theta'(r) = -4r\sum_{n=1}^\infty \alpha_n^{-2}\frac{1}{1+ \frac{r^4}{\alpha_n^4}}, 
\]
where the summation is convergent, once again, due to the growth of $\alpha_n$ at infinity. 
This shows that $\partial_r \mu(r)$  only vanishes at $r=0$. Hence, 
there are no exceptions in the values that $\mu(r)$ attains on $\mathbb{T}$ an infinite number of times, which means that $\rho(z)=1$ infinitely many times on the ray $z=r\exp(i\pi/4)$.

In order to achieve the second part of the theorem we follow Kato \cite[Chapter III, \S 5]{Kato}. It suffices to show that the resolvent $(\Lper-\lambda)^{-1}$ has only algebraically simple poles. For this we employ an expression for the Green's 
function found in \cite[\S3]{2010BLM2}.  We require, in addition to the solution $\phi(x;\lambda)$, 
a second solution $\psi(x;\lambda)$ which is analytic in $\lambda$ for 
$x\neq 0$, satisfies $\psi(\pm \pi;\lambda)=0$, and has the Wronskian property
\[ \left| \begin{array}{cc} \phi(x;\lambda) & \psi(x;\lambda) \\
                                     p(x)\phi'(x;\lambda) & p(x)\psi'(x;\lambda) \end{array}\right| \equiv 1. \]
Here $p$ is the coefficient introduced in (\ref{intfact}). According to \cite[(3.1)-(3.6)]{2010BLM2}, for any $F\in L^2(-\pi,\pi)$, $u=(\Lper-\lambda)^{-1}F$ is given by
\begin{align*}    u(x;\lambda) &=   \phi(x;\lambda)\int_{x}^{\pi}\psi(t;\lambda)\frac{p(t)}{\epsilon f(t)}F(t)dt
  + \psi(x;\lambda) \int_0^x \phi(t;\lambda) \frac{p(t)}{\epsilon f(t)}F(t)dt \\
 & \qquad  + \hspace{5mm}
\frac{\phi(x;\lambda)\phi(-\pi;\lambda)}{\phi(\pi;\lambda)-\phi(-\pi;\lambda)} {\displaystyle \int_{-\pi}^{\pi} \psi(t;\lambda) 
   \frac{p(t)}{\epsilon f(t)}F(t)dt }. 
\end{align*}
The only singularities in this expression come from the zeros of the denominator, which in turn are the zeros of  
\[
    1-\frac{\phi(-\pi;\lambda)}{\phi(\pi;\lambda)}=1-\rho(z).
\]
Indeed, recall from \cite{2010BLM} that $\phi(-\pi;\lambda)=\phi(\pi;-\lambda)$. Moreover, multiplicities coincide except at $z=0$, where a double root in terms of $z$ is a simple root in terms of $\lambda$. 

Apart from the double root at $z=0$,  all the roots of $\rho(z)=1$ are simple. They lie on the lines $\arg(z)\equiv \pi/4$ (modulo $\pi/2$) where $|\rho(z)|=1$, and on these lines the phase of $\rho$ has been shown to be non-stationary, except at $z=0$.
 Thus  $(\Lper-\lambda)^{-1}$ has only simple poles as needed.
\end{proof}

In the sequel we follows \cite[Section~3.3-3.4]{2007Davies} for the notions of conditional and unconditional basis.
Recall that if $\{\phi_n\}$ is a conditional basis
on a Hilbert space, there exists  a dual set $\{\phi_n^*\}$ such that
$\{\phi_n,\phi_n^*\}$ is  bi-orthogonal in the sense that
$\langle \phi_n,\phi_m^*\rangle=\delta_{nm}$ and for all $g\in \mathcal{H}$,
\[
    \|g-\sum_{n=1}^k \widehat{g}(n)\phi_n\|\to 0 \qquad k\to \infty
\] 
where the ``generalised Fourier coefficients'' $\widehat{g}(n)=\langle g,\phi_n^*\rangle$.

\begin{lem} \label{diagonalisation}
Let $M:\dom(M)\longrightarrow \mathcal{H}$ be a closed operator on the  infinite-dimensional Hilbert space $\mathcal{H}$. Assume that the resolvent of $M$ is compact and denote by
$\mu_n\in \Cc$ the eigenvalues of $M$ with corresponding eigenfunctions $\phi_n\not=0$,
$M\phi_n=\mu_n\phi_n$. If $\{\phi_n\}_{n=1}^\infty$ is a conditional basis,
then $\dom(M)=\mathcal{D}$ where
\begin{equation}    \label{domai_compact}
    \mathcal{D}=\{g\in \mathcal{H}:\sum_{n=1}^k \mu_n\widehat{g}(n)\phi_n \to h\in \mathcal{H}\}
\end{equation}
and $Mg=h$ for $g\in \mathcal{D}$.
\end{lem}

\begin{proof}
Let $g\in \mathcal{D}$ and $g_k=\sum_{n=1}^k \widehat{g}(n)\phi_n$. Then
$g_k\in \dom(M)$, $g_k\to g$ and $M(g_k-g_j)\to 0$ as $j,k\to \infty$. Since $M$ is closed, then $g\in \dom(M)$ and $Mg=\lim_{k\to \infty} Mg_k$.

Conversely, let $g\in \dom(M)$. Then $g=(M-\mu)^{-1}z$ for a suitable
$\mu\not=\mu_n$ and $z\in \mathcal{H}$. Note that the spectrum of $(M-\mu)^{-1}$
is $\{\frac{1}{\mu_n-\mu}\}\cup\{0\}$ and $(M-\mu)^{-1}\phi_n=\frac{1}{\mu_n-\mu}\phi_n$.  
Also 
\[
\langle (M^*-\overline{\mu})^{-1}\phi^*_n,\phi_m\rangle=
\langle \phi^*_n,(M-\mu)^{-1}\phi_m\rangle = \frac{\delta_{nm}}{\overline{\mu_n}-\overline{\mu}},
\]
so $(M^*-\overline{\mu})^{-1}\phi_n^*=\frac{1}{\overline{\mu_n}-\overline{\mu}} \phi_n^*$.
Hence
\[
    \widehat{g}(n)= \langle g,\phi^*_n\rangle = \langle (M-\mu)^{-1}z,\phi^*_n\rangle 
    = \frac{\widehat{z}(n)}{\mu_n-\mu}. 
\]
Thus $(\mu_n-\mu)\widehat{g}(n)=\widehat{z}(n)$. Since both 
$\|z-\sum_{n=1}^k\widehat{z}(n)\phi_n\|\to 0$ and $\|\mu g-\mu g_k\|\to 0$, then
$\sum_{n=1}^k\mu_n \widehat{g}(n)\phi_n \to h$ for some $h \in \mathcal{H}$. 
Moreover $h=z+\mu g=Mg$. 
\end{proof}

Assume now that $\{\phi_n\}$ is an unconditional basis of $\mathcal{H}$.
Then the norm generated by $\{\phi_n^*\}$ is equivalent to the norm of $\mathcal{H}$. That is, there are positive constants
$0<a\leq b<\infty$, such that
\begin{equation} \label{equivalent_norms}
     a\|g\|^2 \leq \sum |\hat{g}(n)|^2 \leq b\|g\|^2 \qquad \qquad \forall g\in \mathcal{H}.
\end{equation}
See for example \cite[Theorem~3.4.5]{2007Davies}.
Consider the case $M=\Lper$ and fix the eigenfunctions such that $\Lper \phi_n=i\lambda_n \phi_n$ for $\lambda_n\in \Rr$ the corresponding rotated eigenvalues. 

The following statement is a consequence of Theorem~\ref{singularity_I+}. The hypothesis on the degree of smoothness of $f$ is only present here in order to invoke the latter. Most likely the conclusion will also hold true for less regular $f$, but the treatment of this case will require an analysis beyond the scope of the present paper.

\begin{thm} \label{non-unconditional-basis}  
If $f\in C^7(\Rr)$, the eigenfunctions of $\Lper$ do not form an unconditional basis.
\end{thm}
\begin{proof} Assume, for a contradiction, that $\{\phi_n\}$ is an unconditional basis.
Let 
\[                                                                          
    h(x)=\left\{ \begin{array}{ll}                                           
        4x/\pi                   & 0<x<\pi/4 \\                              
     -8x^2/\pi^2+8x/\pi-1/2       & \pi/4<x<3\pi/4 \\                        
     -4(x-\pi)/\pi & 3\pi/4<x<\pi \\
      \text{extend to an odd function} & -\pi<x<0  .                                        
   \end{array} \right.                                                      
\] 
Then $h'\in  \AC\Ito$, $h\not \in H^3\Ito$ and $h\in \domlper$
(by virtue of the characterisation of $\domlper$ given in Section~\ref{section2}). For $k\in \Nn$, let
\begin{equation*} \label{fake_solution}
      u_k(t,x)=\sum_{n=-k}^k e^{-i \lambda_n t} \widehat{h}(n) \phi_n(x).
\end{equation*}
The fact that $f \in C^7(\Rr)$ together with a bootstrap argument 
yield $\partial_x^8\phi_n \in \ACl((-\pi,0)\cup(0,\pi))$ for all $n\in \mathbb{N}$.
Hence $u_k \in \CT_T$ for all $T>0$. Moreover, $u_k$ is a solution of \eqref{B} with $g=h_k=\sum_{n=-k}^k  \widehat{h}(n) \phi_n$.
Below we show that the limit $u:=\lim_{k\to\infty} u_k$ exists in a suitable sense,  $u\in \CT_T$ for all $T>0$ and $u$ is a solution of \eqref{B} with $g=h$. This would immediately complete the proof,
as Theorem \ref{singularity_I+} then implies the contradictory statement $h\in C^3(-\pi,\pi)$.

Since $h \in L^2(-\pi,\pi)$,
\begin{equation*} \label{l2}
\sum_{n=-\infty}^\infty |\widehat{h}(n)|^2 < \infty.
\end{equation*}
By virtue of  \eqref{equivalent_norms}, it follows that $\{u_k(t,\cdot)\}_k$ is a Cauchy sequence in $L^2 \Ito$ for all $t \in \Rr$.
Then $\{u_k(t,\cdot)\}_k$ converge in $L^2\Ito$ to a limit, $u(t,\cdot)$. This convergence is uniform in $t$, so 
$\{u_k\}_k$ converges in $L^2([0,T]\times \Ito)$ to $u$
 and $u \in C([0,T),L^2 \Ito)$ for all $T>0$. We show that $u\in W^{1,2}_{2,\loc}\big((0,T)\times [\Imi\cup\Ipl]$ in three further steps.

According to our assumption, the representation of $\dom(\Lper)$ given by \eqref{domai_compact} holds true. Then
\begin{equation} \label{dtl2}
\sum_{n=-\infty}^\infty |\mu_n\widehat{h}(n)|^2 < \infty, \qquad \mu_n = i\lambda_n.
\end{equation}
From \eqref{equivalent_norms} and \eqref{dtl2} it follows that $\{\partial_tu_k(t,\cdot)\}_k$ is a Cauchy sequence in $L^2 \Ito$ for all $t \in \Rr$. Then $\{\partial_tu_k\}_k$ converges in $L^2([0,T]\times \Ito)$ to the distributional derivative $\partial_tu$, so that $\partial_tu \in L^2([0,T]\times\Ito)$.

Let us now show that $\partial_xu$ is in $L^2_{\loc}\big((0,T)\times [\Imi\cup\Ipl]\big)$.  
To this end we use the differential equation in \eqref{B}.  Let $\zeta \colon \Rr \times \Rr \to \Rr$ be a smooth function whose support is contained in $(0,T)\times [\Imi\cup\Ipl]$. For any $v \in \CT_T$ satisfying $(\partial_t + L_\varepsilon)(v) = 0$, \begin{equation} \label{cutoff}
(\partial_t + \ell_\varepsilon)(\zeta v) = v(\partial_t + \ell_\varepsilon)(\zeta) + 2\varepsilon f \zeta'v'.
\end{equation}
Multiplying \eqref{cutoff} by $\zeta \overline{v}=\overline{\zeta v}$ and adding the result to its complex conjugate gives
\begin{align*}
\varepsilon(f(\zeta v)')'\zeta \overline{v} + \varepsilon(f(\zeta\overline{v})')'{\zeta v} & = 
2\zeta|v|^2(\partial_t + \ell_\varepsilon)(\zeta) + 2\varepsilon f \zeta'v'\zeta \overline{v}
+ 2\varepsilon f \zeta'\overline{v'}\zeta v \\
& \quad - \partial_t|\zeta v|^2 - \partial_x|\zeta v|^2.
\end{align*} 
Integrating in the spatial variable, integrating by parts and Cauchy-Schwarz inequality, lead to the estimate
\begin{align*}
& \left|2\varepsilon \int_{-\pi}^\pi f(x)|(\zeta v)'(t,x)|^2dx \right| \\
& \leq c_{10}\left(\int_{-\pi}^\pi |v(t,x)|^2dx + \int_{-\pi}^\pi |\partial_tv(t,x)|^2dx + \int_{-\pi}^\pi |(\zeta v)'(t,x)v(t,x)|dx\right) \\
& \leq \left(c_{10}+\frac{c_{10}}{\delta}\right)\int_{-\pi}^\pi |v(t,x)|^2dx + c_{10}\int_{-\pi}^\pi |\partial_tv(t,x)|^2dx + \delta c_{10}\int_{-\pi}^\pi |(\zeta v)'(t,x)|^2dx
\end{align*}
for $\delta>0$.
Here and below $c_j>0$ are constants which only depend on $\varepsilon$, $\zeta$ and $f$. Choosing $\delta>0$ sufficiently small enables us to move the last term on the right-hand side to the left. Integrating in $t$ yields
\begin{equation} \label{est}
\int_0^T \!\!\int_{-\pi}^\pi\!\! |(\zeta v)'(t,x)|^2dx\, dt \leq c_{11}\!\!
\left(\int_0^T \!\! \int_{-\pi}^\pi\!\! |v(t,x)|^2dx\,dt + \int_0^T
\!\! \int_{-\pi}^\pi\!\! |\partial_tv(t,x)|^2dx\,dt\right) \!\!.
\end{equation}
Since $u_k-u_j \in \CT_T$ and $(\partial_t + \ell_\varepsilon)(u_k - u_j) = 0$ for any $k,j \in \Nn$, on applying
 \eqref{est} with $v = u_k - u_j$ we obtain an estimate where $c_{11}$ is independent of $k$ and $j$. 
 Since $\{u_k\}$ and $\{\partial_tu_k\}$ are Cauchy sequences, we conclude 
 that also $\{\partial_x(\zeta u_k)\}_k$ is a Cauchy sequence in $L^2([0,T]\times \Ito)$ for each fixed 
 $\zeta$ whose support is contained in $(0,T)\times [\Imi\cup\Ipl]$. Thus $\partial_xu \in 
 L^2_{\loc}\big((0,T)\times [\Imi\cup\Ipl]\big)$.
 
The equation \eqref{cutoff} implies that $\{\partial_x^2(\zeta u_k)\}_k$ is also a Cauchy sequence and, by an analogous reasoning as before, $\partial_x^2u\in  L^2_{\loc}\big((0,T)\times [\Imi\cup\Ipl]\big)$. 
 Thus $u \in \CT_T$
  for any finite $T > 0$. Moreover, since $u_k$ solves \eqref{B} with initial condition $h_k$, Lemma~\ref{closed} implies that $u$ solves \eqref{B} with initial condition $h$. 
\end{proof}

\section*{Acknowledgements} We kindly acknowledge support from MOPNET,
CANPDE and EPSRC grant 113242.

\end{document}